\documentclass[12pt, a4]{amsart}
\usepackage{amsmath, amssymb, amsthm}
\usepackage{geometry}
\usepackage[T1]{fontenc}
\usepackage{lmodern}
\usepackage{breqn}
\usepackage{adjustbox}
\usepackage{pdflscape}
\usepackage{booktabs} 
\usepackage{enumitem} 
\usepackage{makecell} 
\usepackage{xcolor}
\usepackage[caption=false]{subfig} 

\usepackage{fullpage}

\date{\today}

\newtheorem{theorem}{Theorem}[section]

\newtheorem{definition}[theorem]{Definition}

\newtheorem{proposition}[theorem]{Proposition}

\begin{document}

\title[A fractional thermostat eigenvalue problem]{On a nonlocal fractional thermostat eigenvalue problem}

\date{}

\author[G. Infante]{Gennaro Infante}
\address{Gennaro Infante, Dipartimento di Matematica e Informatica, Universit\`{a} della
Calabria, 87036 Arcavacata di Rende, Cosenza, Italy}%
\email{gennaro.infante@unical.it}%

\author[T. Zeghida]{Takieddine Zeghida}
\address{Takieddine Zeghida, Department of Mathematics, Faculty of Science, University Badji Mokhtar Annaba, 23000, Annaba, Algeria}%
\email{takieddine.zeghida@univ-annaba.dz}%

\begin{abstract}
We study the existence of solutions for a parameter-dependent nonlocal boundary value problem involving a Caputo fractional derivative, which generalizes a classic thermostat model. Our approach extends previous work by considering two nonlinear functionals occurring in the boundary conditions and, crucially, by analyzing cases where the associated Green's function is not necessarily positive and is allowed to change sign. We employ a Birkhoff-Kellogg type theorem in cones to establish the existence of positive eigenvalues with associated eigenfunctions with given norms. Furthermore, we provide explicit intervals that localize the corresponding positive eigenvalues. The applicability of our theoretical framework is illustrated with examples.
\end{abstract}


\maketitle
\section{Introduction}
In this paper we investigate the existence of eigenpairs $(u, \lambda)$ for the nonlocal Caputo fractional boundary value problem (BVP)
\begin{equation}\label{bvp-intro}
\begin{cases}
    {}^{C}\!D^{\alpha}u(t) + \lambda f(t, u(t)) = 0, \quad t \in (0,1),  \\
    u'(0) + \lambda H_1[u] = 0,  \\
    \beta\,{}^{C}\!D^{\alpha-1}u(1) + u(\eta) = \lambda H_2[u]. 
\end{cases}
\end{equation}
Here $1 < \alpha < 2$, $\eta \in (0,1)$, $\beta > 0$, ${}^{C}\!D^{\alpha}$ is the Caputo fractional derivative, $\lambda > 0$ is a positive parameter, $f: [0,1] \times \mathbb{R} \to [0, \infty)$ is a continuous function, and $H_1, H_2: C[0,1] \to [0, \infty)$ are compact non-negative functionals.

The BVP~\eqref{bvp-intro} can be interpreted as a perturbation of the fractional thermostat model developed by Nieto and Pimentel in~\cite{nieto2013pimentel}, namely
\begin{equation}\label{bvp-Nieto}
\begin{cases}
    {}^{C}\!D^{\alpha}u(t) +  f(t, u(t)) = 0, \quad t \in (0,1),  \\
    u'(0) = 
    \beta\,{}^{C}\!D^{\alpha-1}u(1) + u(\eta) = 0,  
\end{cases}
\end{equation}
where the authors put emphasis on cases with a corresponding non-negative Green's function. The methodology in~\cite{nieto2013pimentel} relied on the classical Krasnosel'ski\u{\i}-Guo fixed point theorem.
The unique solvability of the BVP~\eqref{bvp-Nieto} has been investigated by Harjani, L\'{o}pez and Sadarangani~\cite{Harjani2025}, by means of a Rus’s fixed point theorem involving two metrics, by Caballero, Harjani and Sadarangani~\cite{Caballero2020}, via  the Banach contraction principle, and by Bai and Zhaı~\cite{bai2021}, who utilized a fixed point theorem on the sum of two operators. The case of the interaction of two fractional thermostats has been discussed by Infante and Rihani in~\cite{gi-sr}, via fixed point index theory.
The results of Nieto and Pimentel were also complemented by
Cabada and Infante~\cite{cabada2014}, 
who dealt, by classical fixed point index,
with the existence of positive solutions, under boundary conditions of affine type that involve Stieltjes integrals, namely
$$
 u'(0)+\hat{H}[u]= \beta\,{}^{C}\!D^{\alpha-1}u(1)+u(\eta)=0,
$$
where
$$
\hat{H}[u]=\Lambda_0+\int_0^1 u(s)\,d\Lambda(s).
$$
More recently, Etemad and collaborators~\cite{Etemad2021} studied a generalization of the BVP~\eqref{bvp-Nieto}, by considering a more general fractional differential equation and more complex BCs that involve \emph{linear} functionals.

Regarding the existence and non-existence of parameter-dependent positive solutions, the problem
\begin{equation}\label{eigen-nieto}
\begin{cases}
    {}^{C}\!D^{\alpha}u(t) +  \lambda f(t, u(t)) = 0, \quad t \in (0,1), \\
    u'(0) = 
    \beta\,{}^{C}\!D^{\alpha-1}u(1) + u(\eta) = 0, 
\end{cases}
\end{equation}
has been investigated by Shen, Zhou and Yang~\cite{Shen2016}, by means of the Krasnosel'ski\u{\i}-Guo fixed point theorem, and by Hao and Zhang~\cite{Hao2019}, via fixed point index. The existence and uniqueness of positive solutions of the BVP~\eqref{eigen-nieto} has 
been investigated by Garai, Dey and Chanda~\cite{Garai2018} via fixed point theory, while linear eigenvalue problems were studied in ~\cite{Caballero2020, Harjani2025}.

Here we frame the BVP~\eqref{bvp-intro} as a search for an eigenpair~$(u, \lambda)$ for the associated perturbed 
Hammerstein integral equation
\begin{equation} \label{perintro}
    u(t) =\lambda\bigl( H_1[u]\gamma(t) + H_2[u] + \int_0^1 K(t,s) f(s, u(s)) \: ds\bigr).
\end{equation}
We use a version of the Birkhoff-Kellogg theorem in cones, due to Krasnosel'ski\u{\i} and Lady\v{z}enski\u{\i}~\cite{Kra-Lady}, to provide a unified framework for analyzing this generalized model.
Our work extends the above mentioned analyses to more complex scenarios, allowing boundary conditions describing the presence of nonlinear controllers (not necessarily linear or affine) and parameter dependence. We stress that the functional setting chosen for the BCs is quite general and can be used to deal with nonlinear and nonlocal BCs; these are widely studied objects. We refer the reader to  
 the reviews~\cite{Cabada1, Conti, rma, sotiris, Stik, Whyburn} and the papers~\cite{Goodrich3, Goodrich4, kttmna, ktejde, Picone, jw-gi-jlms}.
Along the lines of the earlier work of Infante and Webb~\cite{gijwpems}, we investigate how the sign properties of the Green's function $K$ and of the function $\gamma$ that occur in~\eqref{perintro} depend on the parameters involved. These properties guide the choice of the cone in which we work and, consequently, affect both the assumptions of our existence theorems and the qualitative properties of the eigenfunctions obtained. In the three cases considered, we establish the existence of eigenfunctions that are strictly positive on the whole interval, non-negative on the whole interval and positive on a suitable subinterval, or positive on a suitable subinterval and possibly sign-changing elsewhere, respectively. The theoretical findings are illustrated by means of two examples.

\section{Preliminaries}
We begin by recalling the definition of the Caputo derivative; for its properties we refer to the books~\cite{anast, dieth, pod, samk}. 

\begin{definition}
For a function $y:[0, +\infty)\to \mathbb{R}$,
the  Caputo  derivative of fractional order
  $\alpha>0$ with $\alpha\notin\mathbb{N}$ is given by
$$
{}^{C}\!D^\alpha
y(t)=\frac{1}{\Gamma(n-\alpha)}\int^t_0\frac{y^{(n)}(s)}{(t-s)^{\alpha+1-n}}\,ds, \quad (n-1<\alpha < n, n=[\alpha] + 1),
$$
where $\Gamma$ denotes the Gamma function, that is
$$\Gamma (s)=\int_{0}^{+\infty}x^{s-1}e^{-x} \, dx,$$
and $[\alpha]$ denotes the integer part of a number $\alpha$.
\end{definition}

In~\cite{nieto2013pimentel} it is shown that
the linear BVP
\begin{equation*}\label{bvp-Nieto-lin}
    {}^{C}\!D^{\alpha}u(t) +y(t) = 0, \quad 
    u'(0) = 
    \beta\,{}^{C}\!D^{\alpha-1}u(1) + u(\eta) = 0,  
\end{equation*}
can be rewritten in the form
\begin{equation*}
u(t)=\int_0^1 K(t,s) y(s) \: ds,
\end{equation*}
where the Green's function $K$ is continuous on  $[0,1]^2$ and has the form
\begin{equation*}\label{eq:kernel_K}
    K(t,s) := \beta + \frac{(\eta-s)^{\alpha-1}}{\Gamma(\alpha)}\chi_{[0,\eta]}(s) - \frac{(t-s)^{\alpha-1}}{\Gamma(\alpha)}\chi_{[0,t]}(s), 
\end{equation*}
where $\chi$ is the characteristic function.
Furthermore, it is known, see \cite{cabada2014}, that the affine, decreasing function
\begin{equation*}\label{eq:gamma_func}
\gamma(t) := \frac{\beta}{\Gamma(3-\alpha)} + \eta - t,
\end{equation*}
solves the BVP
\begin{equation*}
    ^{C}D^{\alpha}u(t) = 0, \quad 
    u'(0) +1= 
    \beta\,{}^{C}\!D^{\alpha-1}u(1) + u(\eta) = 0,  
\end{equation*}
while the function identically equal to $1$ solves the BVP
\begin{equation*}
    ^{C}D^{\alpha}u(t) = 0, \quad 
    u'(0) = 0, \quad
    \beta\,{}^{C}\!D^{\alpha-1}u(1) + u(\eta) = 1. 
\end{equation*}
With these ingredients in mind, we rewrite the BVP~\eqref{bvp-intro} as a perturbed Hammerstein integral equation of the form 
\begin{equation}\label{phie}
u (t)= \lambda T(u)(t),\ t\in [0,1],
\end{equation}
where $T: C[0,1] \to C[0,1]$ is given by
\begin{equation} \label{eq:operator_T}
    T(u)(t) := H_1[u]\gamma(t) + H_2[u] + \int_0^1 K(t,s) f(s, u(s)) \: ds.
\end{equation}
\begin{definition}
The solution of the BVP~\eqref{bvp-intro} is understood in the mild sense, that is as a solution of the equation~\eqref{phie}; we refer to the works of Lan and Lin~\cite{lanfrac} and Cicho\'{n} and Salem~\cite{cichon2020salem}, for comments on regularity issues for equations involving Caputo derivatives.
\end{definition}

\begin{proposition}
  The operator $T$ defined in \eqref{eq:operator_T} maps $C[0,1] \to C[0,1]$ and is completely continuous.
\end{proposition}
\begin{proof}
Due to the continuity of the kernel $K$ and $f$, a routine application of the Arzel\`{a}-Ascoli theorem shows
that the operator $F(u)(t)=\int_0^1 K(t,s) f(s, u(s))\,ds$ is completely continuous. The compactness of the perturbation term
$\hat{\Lambda}(u)(t)=H_1[u]\gamma(t) + H_2[u]$ follows from the compactness of 
the functionals $H_1, H_2$ and the continuity of the fixed function $\gamma$.  
\end{proof}
In what follows, the positivity of the functions $\gamma$ and $K$ plays a key part.

\subsection{Positivity parameters}
We make use of the following numbers:
\begin{enumerate}[label=(\roman*)]
    \item \(\displaystyle \beta_K := \frac{(1 - \eta)^{\alpha - 1}}{\Gamma(\alpha)} \). The threshold for the kernel $K$ to be non-negative on \( [0,1] \times [0,1]:= [0,1]^2\). If $\beta \ge \beta_K$ then $K(t,s) \ge 0$ for all $(t,s) \in [0,1]^2$. If $\beta < \beta_K$ then $K(t,s)$ changes sign on $[0,1]^2$, see~\cite{nieto2013pimentel}.
    \item \(\displaystyle \beta_\gamma := (1 - \eta)\Gamma(3 - \alpha) \): The threshold for \( \gamma(t) \) to be non-negative on \( [0,1] \). If $\beta \ge \beta_\gamma$, $\gamma(t) \ge 0$ for all $t \in [0,1]$. If $\beta < \beta_\gamma$, $\gamma(t)$ changes sign on $[0,1]$. 
    \item \( t_K := \eta + (\beta \Gamma(\alpha))^{\frac{1}{\alpha - 1}} \): The upper limit for \( t \) such that \( K(t,s) \ge 0 \) for all \( s \in [0,1] \).
    \item \( t_{\gamma}: = \eta + \frac{\beta}{\Gamma(3 - \alpha)} \): The upper limit for \( t \) such that \( \gamma(t) \ge 0 \).
\end{enumerate}
Both $K(t,s)$, for every $s\in[0,1]$, and $\gamma(t)$ are non-negative for $t\in[0,\min\{1,t^*\}]$, where $t^*=\min\{t_K,t_\gamma\}$.
Note that
if $\beta \ge \beta_K$, then we have $t_K \ge 1$. If $\beta \ge \beta_\gamma$, then we get $t_{\gamma} \ge 1$.
Thus, if $\beta \ge \max\{\beta_K, \beta_\gamma\}$, we obtain $t^* \ge 1$, which implies the non-negativity on the entire interval $[0,1]$.

\subsection{Upper and lower bounds for the kernel $K$}
In our reasoning the sign and bounding properties of the kernel $K$ are crucial and depend on the parameter $\beta$. With these data we shall construct the corresponding cones. We require the existence of an interval 
 $[a,b] \subseteq [0,1]$, a function $\Phi \in  L^1[0,1]$ and a constant $c_{i,K} \in (0,1]$ such that
    \begin{align*}
    |K(t,s)| \leq \Phi(s), \ \text{for}&\ t,s \in [0,1], \\
    K(t,s) \geq c_{i,K} \Phi(s), \ \text{for}&\ t \in [a,b]\ \text{and}\ s \in
  [0,1].
    \end{align*}
    The specific forms of $\Phi(s)$ and $c_{i,K}$, and the choice of $[a,b]$, depend on the sign of $K(t,s)$; these were studied  in~\cite{nieto2013pimentel} and are presented in Table~\ref{tableK}.
\begin{table}
\centering
\begin{tabular}{| @{} >{\centering\arraybackslash}p{3cm}@{\hspace{0.2cm}}| >{\centering\arraybackslash}p{5.5cm} | >{\centering\arraybackslash}p{6.5cm} @{}|}
\toprule
\textbf{Sign of $K$}   & \textbf{\(\Phi(s)\)} & \textbf{ \(c_{i,K}\)}  \\
\midrule
\(\beta > \beta_K\),\, $K>0$  & \(\beta + \frac{\eta^{\alpha-1}}{\Gamma(\alpha)}\) & \( c_{1,K}=\frac{\beta\Gamma(\alpha)-(1-\eta)^{\alpha-1}}{\beta\Gamma(\alpha)+\eta^{\alpha-1}} \), \newline for $[a,b]=[0,1]$ \\
\midrule
\(\beta = \beta_K\),\, $K\geq 0$,\newline \(K(1,\eta)=0\) & \(\beta + \frac{\eta^{\alpha-1}}{\Gamma(\alpha)}\) & \( c_{2,K}= \frac{\beta\Gamma(\alpha)-(b-\eta)^{\alpha-1}}{\beta\Gamma(\alpha)+\eta^{\alpha-1}} \), \newline for $[a,b]=[0,b]$ with \(\eta \le b<1\)  \\
\midrule
\(\beta < \beta_K\),\, $K \gtreqless 0$ & \(\max\biggl\{\frac{\beta\Gamma(\alpha)+\eta^{\alpha-1}}{\Gamma(\alpha)}, \frac{(1-\eta)^{\alpha-1}-\beta\Gamma(\alpha)}{\Gamma(\alpha)} \biggr\}\) & \( c_{3,K}=\min \biggl\{ \frac{\beta\Gamma(\alpha)-(b-\eta)^{\alpha-1}}{\beta \Gamma(\alpha)+ \eta^{\alpha-1}}, \frac{\beta\Gamma(\alpha)-(b-\eta)^{\alpha-1}}{(1-\eta)^{\alpha-1}-\beta \Gamma(\alpha)} \biggr\}\), \newline for $[a,b]=[0,b]$, where $\eta \le b < 1$ and $\beta\Gamma(\alpha) > (b-\eta)^{\alpha-1}$ \\
\bottomrule
\end{tabular}
\caption{Sign properties of $K$.}
\label{tableK}
\end{table}

\subsection{Upper and lower bounds for the function $\gamma$}
For the function $\gamma$ we seek growth conditions similar to the ones of the kernel $K$; that is, for the interval 
 $[a,b] \subseteq [0,1]$, we seek a constant $\sigma_{i,\gamma} \in (0,1]$ such that 
$$\min_{t \in [a,b]}\gamma(t) \geq \sigma_{i,\gamma} \|\gamma\|_\infty.$$
Since the function $\gamma(t) = \frac{\beta}{\Gamma(3-\alpha)} + \eta - t$ is decreasing and remains positive
in the interval $[0,t_{\gamma})$, by direct calculation we obtain the results presented in Table~\ref{tablegam}.
\begin{table}
\begin{tabular}{|c|c|c|}
\toprule
 \textbf{Sign of $\gamma$} & \textbf{Choice of $[a,b]$} & \textbf{$\sigma_{i,\gamma}$} \\
\midrule
 $\beta > \beta_\gamma$, $\gamma>0$  &  $[a,b]=[0,1]$ & 
$\displaystyle \sigma_{1,\gamma} = \frac{\beta+(\eta-1)\Gamma(3-\alpha)}{\beta+\eta\Gamma(3-\alpha)}$ \\[6pt]
\midrule
 $\beta = \beta_\gamma$, $\gamma\geq 0$ &  $[a,b]=[0,b],\; b \in [\eta,1)$ & 
$\displaystyle \sigma_{2,\gamma} =  \frac{\beta+(\eta-b)\Gamma(3-\alpha)}{\beta+\eta\Gamma(3-\alpha)}$ \\[6pt]
\midrule
 $\beta < \beta_\gamma\, \gamma \gtreqless 0$&  $[a,b]\subset [0,t^*)$ & 
$\displaystyle \sigma_{3,\gamma} =  \frac{\beta+(\eta - b)\Gamma(3-\alpha)}{\max\!\left\{\beta+\eta\Gamma(3-\alpha),\; (1 -\eta)\Gamma(3-\alpha)- \beta\right\}}$ \\[6pt]
\bottomrule
\end{tabular}
\caption{Sign properties of $\gamma$.}\label{tablegam}
\end{table}

\section{Existence and localization of the eigenfunctions}
We use the following version of Birkhoff-Kellogg Theorem in cones due to
Krasnosel'ski\u{\i} and Lady\v{z}enski\u{\i}~\cite{Kra-Lady}, cf.~\cite[Theorem~5.5]{Krasno}.

\begin{theorem}[Birkhoff-Kellogg type Theorem in cones]\label{thm:bk}
Let $(X, \|\cdot\|)$ be a real Banach space, $\hat{K} \subset X$ be a cone, and $U \subset X$ be a bounded open neighborhood of the origin $0$. Let $\hat{T}: \hat{K} \cap \bar{U} \to \hat{K}$ be a compact operator such that
\[
\inf_{x \in \hat{K} \cap \partial U} \|\hat{T}x\| > 0.
\]
Then there exist $\lambda_0 > 0$ and $x_0 \in \hat{K} \cap \partial U$ such that $x_0 = \lambda_0 \hat{T} x_0$.
\end{theorem}
In order to apply this result, we work in the Banach space
$C[0,1]$, endowed with the usual supremum norm. For our growth estimates, we make use of the open bounded set $$\Omega_\rho = \{ u \in C[0,1] : ||u||_\infty < \rho \}.$$
In what follows we keep in mind the data present in the Tables~\ref{tableK},~\ref{tablegam} and build, accordingly, three different cones that have specific positivity properties.

\subsection{Case 1: $\beta > \max\{\beta_K, \beta_\gamma\}$}
Under this condition, we have that the kernel $K$ is strictly positive on $[0,1]^2$ and $\gamma(t) > 0$ on $[0,1]$. We define the cone 
\begin{equation} \label{eq:cone_P1_final_v4}
    P_{\sigma_1} := \Bigl\{ u \in C[0,1] \mid \min_{t \in [0,1]} u(t) \ge \sigma_1 ||u||_\infty \Bigr\},
\end{equation}
with $\sigma_1 = \min\{\sigma_{1,\gamma}, 1, c_{1,K}\}$, where $c_{1,K}$ and $\sigma_{1,\gamma}$ are as in the Tables~\ref{tableK},~\ref{tablegam}.

\begin{theorem}\label{thm:exist_case_A_combined}
Assume $\beta > \max\{\beta_K, \beta_\gamma\}$. Let $\rho > 0$ and assume the following conditions.
\begin{itemize}
    \item There exists a function $\underline{\delta}_\rho \in C([0,1],\mathbb{R}^+)$ such that $f(t, u) \ge \underline{\delta}_\rho(t)$ for all $(t,u) \in [0,1]\times [\sigma_1 \rho, \rho]$.
    \item  There exist $\underline{\eta}_{1,\rho}, \underline{\eta}_{2,\rho} \ge 0$ such that
\[
H_1[u] \ge \underline{\eta}_{1,\rho} \ \text{and} \ H_2[u] \ge \underline{\eta}_{2,\rho}, \ \text{for every}\ u \in P_{\sigma_1} \cap \partial \Omega_\rho.
\]
    \item The following algebraic inequality holds:
    \[
    \gamma(0) \underline{\eta}_{1,\rho} + \underline{\eta}_{2,\rho} + \int_0^1 K(0,s)\underline{\delta}_\rho(s)\,ds > 0.
    \]
\end{itemize}
Then there exist $\lambda_\rho >0$ and $u_\rho \in P_{\sigma_1} \cap \partial{\Omega_\rho}$ such that the couple $(\lambda_\rho,u_\rho)$ 
solves the BVP~\eqref{bvp-intro} and 
$\sigma_1 \rho \leq u_\rho(t) \leq \rho$ for every $t\in [0,1]$.
\end{theorem}

\begin{proof}
We consider the completely continuous operator $T$ defined in \eqref{eq:operator_T} on the cone $P_{\sigma_1}$. The proof proceeds in two main steps.

\textbf{Step 1: Cone Invariance ($T(P_{\sigma_1} \cap \overline{\Omega_\rho}) \subset P_{\sigma_1}$)}.
Let $u \in P_{\sigma_1} \cap \overline{\Omega_\rho}$. We must show that $v = Tu$ belongs to $P_{\sigma_1}$.

Under the given condition on $\beta$, we have $\gamma(t) > 0$ and $K(t,s) > 0$. By the definitions of the constants $\sigma_{1,\gamma}$ and $c_{1,K}$, for every $t \in [0,1]$ we have the bounds $\gamma(t) \ge \sigma_{1,\gamma} ||\gamma||_\infty$ and $K(t,s) \ge c_{1,K} \Phi(s)$ for all $s \in [0,1]$.
Now note that, for every  $t \in [0,1]$, since $H_1, H_2$, and $f$ are non-negative we have
\begin{align*}
    v(t) :=& H_1[u]\gamma(t) + H_2[u] + \int_0^1 K(t,s) f(s, u(s))\,ds \\
    &\ge H_1[u](\sigma_{1,\gamma} ||\gamma||_\infty) + H_2[u] \cdot 1 + \int_0^1 (c_{1,K} \Phi(s))f(s,u(s))\,ds,
\end{align*}
which, in turn, implies
\[
\min_{t \in [0,1]} v(t) \ge \sigma_1 \left( H_1[u]||\gamma||_\infty + H_2[u] + \int_0^1 \Phi(s)f(s,u(s))\,ds \right).
\]
Now note that
\[
||v||_\infty = \max_{t \in [0,1]} v(t) \le H_1[u]||\gamma||_\infty + H_2[u] + \int_0^1 \Phi(s)f(s,u(s))\,ds,
\]
which yields 
\[
\min_{t \in [0,1]} v(t) \ge \sigma_1 ||v||_\infty.
\]
Thus the operator $T$ maps the set $P_{\sigma_1} \cap \overline{\Omega_\rho}$ into the cone $P_{\sigma_1}$.

\textbf{Step 2: Verification of the growth condition}.

We now show that $\inf_{u \in P_{\sigma_1} \cap \partial \Omega_\rho} ||Tu||_\infty > 0$.
Take $u \in P_{\sigma_1} \cap \partial \Omega_\rho$. This implies $||u||_\infty = \rho$ and therefore $u(t) \in [\sigma_1 \rho, \rho]$ for every $t \in [0,1]$.
In this case, both $\gamma$ and the kernel $K$ (for fixed $s$) are positive and monotonically decreasing with respect to $t$ on $[0,1]$. This implies that
\[
||Tu||_\infty = \max_{t \in [0,1]} Tu(t) = Tu(0).
\]
Using the lower bounds on $H_1, H_2$, and $f$ for $u \in P_{\sigma_1} \cap \partial \Omega_\rho$:
\begin{align*}
    ||Tu||_\infty = & H_1[u]\gamma(0) + H_2[u] + \int_0^1 K(0,s)f(s,u(s))\,ds \\
    &\ge \underline{\eta}_{1,\rho}\gamma(0) + \underline{\eta}_{2,\rho} + \int_0^1 K(0,s)\underline{\delta}_\rho(s)\,ds.
\end{align*}
By assumption, the right-hand side is a fixed positive number that is independent of the particular choice of $u$.  This establishes that
\[
\inf_{u \in P_{\sigma_1} \cap \partial \Omega_\rho} ||Tu||_\infty > 0.
\]
Since $T: P_{\sigma_1} \cap \overline{\Omega_\rho} \to P_{\sigma_1}$ is completely continuous and satisfies the conditions of Theorem~\ref{thm:bk}, there exists an eigenfunction $u_\rho \in P_{\sigma_1}$ with $||u_\rho||_\infty = \rho$ and a corresponding positive eigenvalue $\lambda_\rho > 0$.
\end{proof}

\subsection{Case 2: $\beta = \max\{\beta_K, \beta_\gamma\}$}
This case addresses scenarios where the operator $T$ remains non-negative on $[0,1]$, but $K(t,s)$ or $\gamma(t)$ can be zero at $t=1$ or $(t,s)=(1,\eta)$. Specifically:
\begin{itemize}
    \item If $\beta = \beta_K$, then $K(1,\eta)=0$.
    \item If $\beta = \beta_\gamma$, then $\gamma(1)=0$.
\end{itemize}
In these situations, $t^*=1$. We fix $b \in [\eta,1)$ and define the cone
\begin{equation} \label{eq:cone_P2_caseB}
    P_{\sigma_2}: = \Bigl\{ u \in C[0,1] \mid u\geq 0,\min_{t \in [0,b]} u(t) \ge \sigma_2 ||u||_\infty \Bigr\},
\end{equation}
with $\sigma_2 = \min\{\sigma_{2,\gamma}, 1, c_{2,K}\}$, where $c_{2,K}$ and $\sigma_{2,\gamma}$ are as in the Tables~\ref{tableK},~\ref{tablegam}.
Note that the expressions for $c_{2,K}$ and $\sigma_{2,\gamma}$ given in Tables~\ref{tableK} and~\ref{tablegam} remain valid under the conditions $\beta\geq\beta_K$ and $\beta\geq\beta_\gamma$, respectively.

\begin{theorem}\label{thm:exist_case_B_combined}
Assume $\beta = \max\{\beta_K, \beta_\gamma\}$. 
Let $\rho > 0$ and assume the following conditions.
\begin{itemize}
    \item There exists a function $\underline{\delta}_\rho \in C([0,b],\mathbb{R}^+)$ such that $f(t, u) \ge \underline{\delta}_\rho(t)$ for all $(t,u) \in [0,b]\times [\sigma_2 \rho, \rho]$.
    \item  There exist $\underline{\eta}_{1,\rho}, \underline{\eta}_{2,\rho} \ge 0$ such that
\[
H_1[u] \ge \underline{\eta}_{1,\rho} \ \text{and} \ H_2[u] \ge \underline{\eta}_{2,\rho}, \ \text{for every}\ u \in P_{\sigma_2} \cap \partial \Omega_\rho.
\]
    \item The following algebraic inequality holds:
    \[
    \gamma(0) \underline{\eta}_{1,\rho} + \underline{\eta}_{2,\rho} + \int_0^b K(0,s)\underline{\delta}_\rho(s)\,ds > 0.
    \]
\end{itemize}
Then there exist $\lambda_\rho >0$ and $u_\rho \in P_{\sigma_2} \cap \partial \Omega_\rho$ such that the couple $(\lambda_\rho,u_\rho)$ 
solves the BVP~\eqref{bvp-intro}, $u_\rho (t)\geq 0$ for every $t\in [0,1]$, and
$\sigma_2 \rho \leq u_\rho(t) \leq \rho$ for every $t\in [0,b]$.
\end{theorem}

\begin{proof}
We proceed in a similar way as in the previous case.

\textbf{Step 1: Cone invariance}.
Take $u \in P_{\sigma_2} \cap \overline{\Omega_\rho}$. Let $v=Tu$. Since $\gamma$, $K$, $H_1$, $H_2$ and $f$ are non-negative, $v(t)\geq 0$ for every $t\in[0,1]$.
Note that $$||v||_\infty \le H_1[u]||\gamma||_\infty + H_2[u] + \int_0^1 \Phi(s)f(s,u(s))\,ds.$$
For $t\in [0,b]$, we have
\[
v(t) \ge H_1[u](\sigma_{2,\gamma} ||\gamma||_\infty) + H_2[u] + c_{2,K} \int_0^1 \Phi(s)f(s,u(s))\,ds,
\]
which implies 
\[
\min_{t \in [0,b]} v(t) \ge \sigma_2 ||v||_\infty.
\]
Thus we have $Tu \in P_{\sigma_2}$.

\textbf{Step 2: Verification of the growth condition}.

Let $u \in P_{\sigma_2} \cap \partial \Omega_\rho$. Then  $||u||_\infty= \rho$ and $u(t)\geq \sigma_2 ||u||_\infty$ for every $t\in [0,b]$.
The conditions on $\beta$ ensure that $\gamma(t)$ and $K(t,s)$ are non-negative and decreasing in $t$ on $[0,1]$. Therefore we have
\begin{multline*}
    ||Tu||_\infty=Tu(0) = H_1[u]\gamma(0) + H_2[u] + \int_0^1 K(0,s)f(s,u(s))\,ds \\
    \ge H_1[u]\gamma(0) + H_2[u] + \int_0^b K(0,s)f(s,u(s))\,ds\ge \underline{\eta}_{1,\rho}\gamma(0) + \underline{\eta}_{2,\rho} + \int_0^b K(0,s)\underline{\delta}_\rho(s)\,ds.
\end{multline*}
This lower bound is a fixed positive constant, independent of the choice of $u$. Therefore we have
\[
\inf_{u \in P_{\sigma_2} \cap \partial \Omega_\rho} ||Tu||_\infty > 0.
\]
The conclusion then follows from Theorem~\ref{thm:bk}.
\end{proof}

\subsection{Case 3: $ \beta < \min\{\beta_K, \beta_\gamma\}$}

In this case, both the kernel $K(t,s)$ and the function $\gamma(t)$ change sign on their respective domains. However, they remain non-negative on an initial sub-interval $[0, t^*]$, where $t^* = \min(t_K, t_{\gamma}) \in (0,1)$. We fix $\eta\leq b<t^*$ and consider the cone 
\begin{equation} \label{eq:cone_P3_caseC}
    P_{\sigma_3} = \Bigl\{ u \in C[0,1] \mid \min_{t \in [0,b]} u(t) \ge \sigma_3 ||u||_\infty \Bigr\},
\end{equation}
with $\sigma_3 = \min\{\sigma_{3,\gamma}, 1, c_{3,K}\}$, where $c_{3,K}$ and $\sigma_{3,\gamma}$ are as in the Tables~\ref{tableK},~\ref{tablegam}.
Note that nonzero functions in $P_{\sigma_3}$ are positive on $[0,b]$ and are allowed to change sign on $[0,1]$,
a type of cone introduced by Infante and Webb in \cite{gijwjiea}
and is similar to
a cone of \emph{non-negative} functions
first used by Krasnosel'ski\u\i{}, see e.g.~\cite{krzab}, and D.~Guo, see e.g.~\cite{Guo1988}. 

\begin{theorem}\label{thm:exist_case_C_combined}
Assume $\beta < \min\{\beta_K, \beta_\gamma\}$. 
Let $\rho > 0$ and assume the following conditions.
\begin{itemize}
    \item There exists a function $\underline{\delta}_\rho \in C([0,b],\mathbb{R}^+)$ such that $f(t, u) \ge \underline{\delta}_\rho(t)$ for all $(t,u) \in [0,b]\times [\sigma_3 \rho, \rho]$.
    \item  There exist $\underline{\eta}_{1,\rho}, \underline{\eta}_{2,\rho} \ge 0$ such that
\[
H_1[u] \ge \underline{\eta}_{1,\rho} \ \text{and} \ H_2[u] \ge \underline{\eta}_{2,\rho}, \ \text{for every}\ u \in P_{\sigma_3} \cap \partial \Omega_\rho.
\]
    \item The following algebraic inequality holds:
    \[
    \gamma(0) \underline{\eta}_{1,\rho} + \underline{\eta}_{2,\rho} + \int_0^b K(0,s)\underline{\delta}_\rho(s)\,ds > 0.
    \]
\end{itemize}
Then there exist $\lambda_\rho >0$ and $u_\rho \in P_{\sigma_3} \cap \partial \Omega_\rho$ such that the couple $(\lambda_\rho,u_\rho)$ 
solves the BVP~\eqref{bvp-intro},  and
$\sigma_3 \rho \leq u_\rho(t) \leq \rho$ for every $t\in [0,b]$.
\end{theorem}

\begin{proof}
The proof relies on applying the Birkhoff-Kellogg type theorem to the operator $T$ on the cone $P_{\sigma_3}$.

\textbf{Step 1: Cone Invariance ($T(P_{\sigma_3} \cap \overline{\Omega_\rho}) \subset P_{\sigma_3}$)}.

Let $u \in P_{\sigma_3} \cap \overline{\Omega_\rho}$. We show that $v = Tu$ is in $P_{\sigma_3}$.
Recall that for $t \in [0,b] \subset [0,t^*)$, both $\gamma(t)$ and $K(t,s)$ are non-negative and we have $\gamma(t) \ge \gamma(b) = \sigma_{3,\gamma}\|\gamma\|_\infty$ and $K(t,s) \ge c_{3,K}\Phi(s)$. Thus we obtain
\begin{multline*}
    \min_{t \in [0,b]} v(t) \ge H_1[u]\,\sigma_{3,\gamma}\,\|\gamma\|_\infty + H_2[u] + c_{3,K} \int_0^1 \Phi(s)f(s,u(s))\,ds\\
    \geq \sigma_3 \left( H_1[u]\,\|\gamma\|_\infty + H_2[u] + \int_0^1 \Phi(s)f(s,u(s))\,ds \right),
\end{multline*}
that, combined with the estimate
\begin{equation*} \label{eq:norm_v_upper_C}
||v||_\infty \le H_1[u]||\gamma||_\infty + H_2[u] + \int_0^1 \Phi(s)f(s,u(s))\,ds,
\end{equation*}
yields $Tu \in P_{\sigma_3}$.

\textbf{Step 2: Verification of the growth condition}.

Let $u \in P_{\sigma_3} \cap \partial \Omega_\rho$. This implies $||u||_\infty = \rho$ and $\min_{t \in [0,b]} u(t) \ge \sigma_3 \rho$.
Note that 
\[
||Tu||_\infty \ge \sup_{t \in [0,b]} Tu(t) = Tu(0).
\]
We can now establish a strictly positive lower bound using the hypotheses:
\begin{multline*}
    ||Tu||_\infty \ge Tu(0) = H_1[u]\gamma(0) + H_2[u] + \int_0^1 K(0,s)f(s,u(s))\,ds   \\
    \geq H_1[u]\gamma(0) + H_2[u] + \int_0^b K(0,s)f(s,u(s))\,ds \ge \underline{\eta}_{1,\rho}\gamma(0) + \underline{\eta}_{2,\rho} + \int_0^b K(0,s)\underline{\delta}_\rho(s)\,ds.
\end{multline*}
This lower bound is independent of the choice of $u$. Therefore we have
\[
\inf_{u \in P_{\sigma_3} \cap \partial \Omega_\rho} ||Tu||_\infty > 0.
\]
The conclusion then follows from Theorem~\ref{thm:bk}.
\end{proof}

\section{Localization of the eigenvalues}

In this section, we provide bounds for the eigenvalues obtained via the Birkhoff-Kellogg type theorem. We present a unified theorem that covers all three cases.

\begin{theorem}[Localization of the eigenvalues] \label{thm:eigenvalue_bounds}
Let the hypotheses of one of the existence theorems (for Case 1, 2, or 3) hold, guaranteeing an eigenpair $(u_\rho, \lambda_\rho)$ with $\lambda_\rho>0$ and $||u_\rho||_\infty = \rho$. In addition assume that the following conditions hold.
\begin{itemize}
    \item There exists $\overline{\delta}_\rho \in C([0,1], \mathbb{R}^+)$ such that $f(t, u) \le \overline{\delta}_\rho(t)$ for every $t \in [0,1]$ and for every $u \in [\tau_i\rho, \rho]$, where 
    $$\tau_i=
    \begin{cases}
    \sigma_1,\ \text{in Case 1},  \\
   0,\ \text{in Case 2},  \\
   -1,\ \text{in Case 3}.  
\end{cases}$$
    \item $H_1[u] \le \overline{\eta}_{1,\rho}$ and $H_2[u] \le \overline{\eta}_{2,\rho}$ for every  $u \in  P_{\sigma_i}\cap \partial  \Omega_\rho.$
\end{itemize}
Then the eigenvalue $\lambda_\rho$ is localized in the interval $[L(\rho), U(\rho)]$, where
\begin{align*}
L(\rho) &:= \frac{\rho}{\overline{\eta}_{1,\rho}||\gamma||_\infty + \overline{\eta}_{2,\rho} + \int_0^1 \Phi(s)\overline{\delta}_\rho(s)\,ds}, \\
U(\rho) &:= \frac{\rho}{\underline{\eta}_{1,\rho}\gamma(0) + \underline{\eta}_{2,\rho} + \int_0^b K(0,s)\underline{\delta}_\rho(s)\,ds},
\end{align*}
and $b=1$ in Case 1.
\end{theorem}
\begin{proof}
Suppose that there exists an eigenpair $(u_\rho, \lambda_\rho)$ provided by one of the existence theorems. Then we have
\begin{multline}\label{ee1}
    \rho=\lambda_\rho\|Tu_\rho\|_\infty\geq \lambda_\rho |Tu_\rho(0) |\geq  \lambda_\rho \Bigl( H_1[u_\rho]\gamma(0) + H_2[u_\rho] + \int_0^b K(0,s)f(s,u_\rho(s))\,ds\Bigr)\\
\geq \lambda_\rho \Bigl(\underline{\eta}_{1,\rho}\gamma(0) + \underline{\eta}_{2,\rho} + \int_0^b K(0,s)\underline{\delta}_\rho(s)\,ds\Bigr).
\end{multline}
On the other hand,
\begin{multline}\label{ee2}
    \rho=\lambda_\rho||Tu_\rho||_\infty \le \lambda_\rho \Bigl(H_1[u_\rho]||\gamma||_\infty + H_2[u_\rho] + \int_0^1 \sup_{t \in [0,1]}|K(t,s)| f(s,u_\rho(s))\,ds \Bigr)\\
    \le \lambda_\rho \Bigl(\overline{\eta}_{1,\rho}||\gamma||_\infty + \overline{\eta}_{2,\rho} + \int_0^1 \Phi(s)\overline{\delta}_\rho(s)\,ds\Bigr).
\end{multline}
The inequalities \eqref{ee1}-\eqref{ee2} provide the desired estimate $L(\rho)\leq \lambda_\rho \leq U(\rho)$.
\end{proof}

\section{Illustrative examples and numerical localization of the eigenvalues}\label{five}

We now apply Theorem \ref{thm:eigenvalue_bounds} to two concrete examples, one for the non-negative case (Case~2) and one for the sign-changing case (Case~3), to illustrate how the eigenpair $(\lambda_\rho, u_\rho)$ can be localized.

\subsection{Example for Case 2: Non-Negative Kernel}
Let $\alpha = 1.8$ and $\eta = 0.6$. The critical value for $\gamma(t)$ to be non-negative is $\beta_\gamma = (1-\eta)\Gamma(3-\alpha) = 0.4 \cdot \Gamma(1.2) \approx 0.367$. The critical value for $K(t,s)$ is $\beta_K = (1-\eta)^{\alpha-1}/\Gamma(\alpha) = (0.4)^{0.8}/\Gamma(1.8) \approx 0.516$. We choose $\beta = \beta_K \approx 0.516$ to be in Case 2.

Consider the BVP:
\begin{align*}
    {}^{C}\!D^{1.8}u(t) + \lambda \left( t^2 + \frac{u^2(t)}{1+u^2(t)} \right) &= 0, \quad t \in (0,1), \\
    u'(0) + \lambda \left( \frac{|u(0.2)|}{1+||u||_\infty} \right) &= 0, \\
    \beta_K\,{}^{C}\!D^{0.8}u(1) + u(0.6) &= \lambda \left( \frac{1}{10} \int_0^1 u^2(s)\,ds \right).
\end{align*}
For a solution with $||u||_\infty = \rho$, we can identify the bounds on the nonlinearities:
\begin{itemize}
    \item $f(t,u) = t^2 + u^2/(1+u^2)$. Thus, $\underline{\delta}_\rho(t) = t^2$ and $\overline{\delta}_\rho(t) = t^2 + \rho^2/(1+\rho^2)$.
    \item $H_1[u] = |u(0.2)|/(1+||u||_\infty)$. Thus, $\underline{\eta}_{1,\rho} = 0$ and $\overline{\eta}_{1,\rho} = \rho/(1+\rho)$.
    \item $H_2[u] = \frac{1}{10}\int_0^1 u(s)^2\,ds$. Thus, $\underline{\eta}_{2,\rho} = 0$ and $\overline{\eta}_{2,\rho} = \rho^2/10$.
\end{itemize}
The eigenpair exists by Theorem~\ref{thm:exist_case_B_combined}, for any $b \in [\eta,1)$; we take $b=0.9$. Here $||\gamma||_\infty=\gamma(0)\approx 1.162$ and, by Table~\ref{tableK}, $\Phi\equiv \beta+\eta^{\alpha-1}/\Gamma(\alpha)\approx 1.229$. Applying the formulas from Theorem \ref{thm:eigenvalue_bounds}, we obtain
\[
\frac{\rho}{1.162\,\frac{\rho}{1+\rho}+\frac{\rho^2}{10}+1.230\left(\frac13+\frac{\rho^2}{1+\rho^2}\right)}
\le \lambda_\rho \le \frac{\rho}{\int_0^{0.9}K(0,s)s^2\,ds}\approx 7.07\,\rho .
\]
\subsection{Example for Case 3: Sign-Changing Kernel}
Let $\alpha = 1.5$ and $\eta = 0.5$. The critical values are $\beta_K \approx 0.798$ and $\beta_\gamma \approx 0.443$. We choose $\beta = 0.1$ to ensure that we are in Case~3.

Consider the BVP:
\begin{align*}
    {}^{C}\!D^{1.5}u(t) + \lambda \left( \frac{t}{1+u^2(t)} \right) &= 0, \quad t \in (0,1), \\
    u'(0) + \lambda \left( \frac{|u(0.25)|}{1+||u||_\infty} \right) &= 0, \\
    0.1\,{}^{C}\!D^{0.5}u(1) + u(0.5) &= \lambda \left( \frac{1}{20} \int_0^1 u^2(s)\,ds \right).
\end{align*}
For a solution with $||u||_\infty = \rho$, we identify the bounds:
\begin{itemize}
    \item $f(t,u) = t/(1+u^2)$. Thus, $\underline{\delta}_\rho(t) = t/(1+\rho^2)$ and $\overline{\delta}_\rho(t) = t$.
    \item $H_1[u] = |u(0.25)|/(1+||u||_\infty)$. Thus, $\underline{\eta}_{1,\rho} = 0$ and $\overline{\eta}_{1,\rho} = \rho/(1+\rho)$.
    \item $H_2[u] = \frac{1}{20}\int_0^1 u^2 (s)\,ds$. Thus, $\underline{\eta}_{2,\rho} = 0$ and $\overline{\eta}_{2,\rho} = \rho^2/20$.
\end{itemize}
Here $t^*=\min\{t_K,t_\gamma\}\approx 0.5079$, so we may take $b=0.5$. Then Theorem~\ref{thm:exist_case_C_combined} provides existence of the eigenpairs. Moreover $||\gamma||_\infty=\gamma(0)\approx 0.613$ and, by Table~\ref{tableK}, $\Phi\approx 0.898$. Applying the formulas from Theorem~\ref{thm:eigenvalue_bounds}, we obtain
\[
\frac{\rho}{0.613\,\frac{\rho}{1+\rho}+\frac{\rho^2}{20}+0.449}
\le \lambda_\rho \le \frac{\rho\,(1+\rho^2)}{\int_0^{0.5}K(0,s)s\,ds}\approx 15.22\,\rho\,(1+\rho^2).
\]
Applying the formulas from Theorem~\ref{thm:eigenvalue_bounds}, we compute the functions $L(\rho)$ and $U(\rho)$ for the two examples and plot, by means of Python, the localization regions for $(\rho,\lambda_{\rho})$. This is illustrated in Figure~\ref{fig:exampleplot_positive}, where the points $(\rho,\lambda_\rho)$ are not allowed outside the shaded regions.

\begin{figure}[htbp]
    \centering
    \subfloat[Non-negative kernel.]{%
        \includegraphics[width=0.48\textwidth,trim=0 0 406.4bp 0,clip]{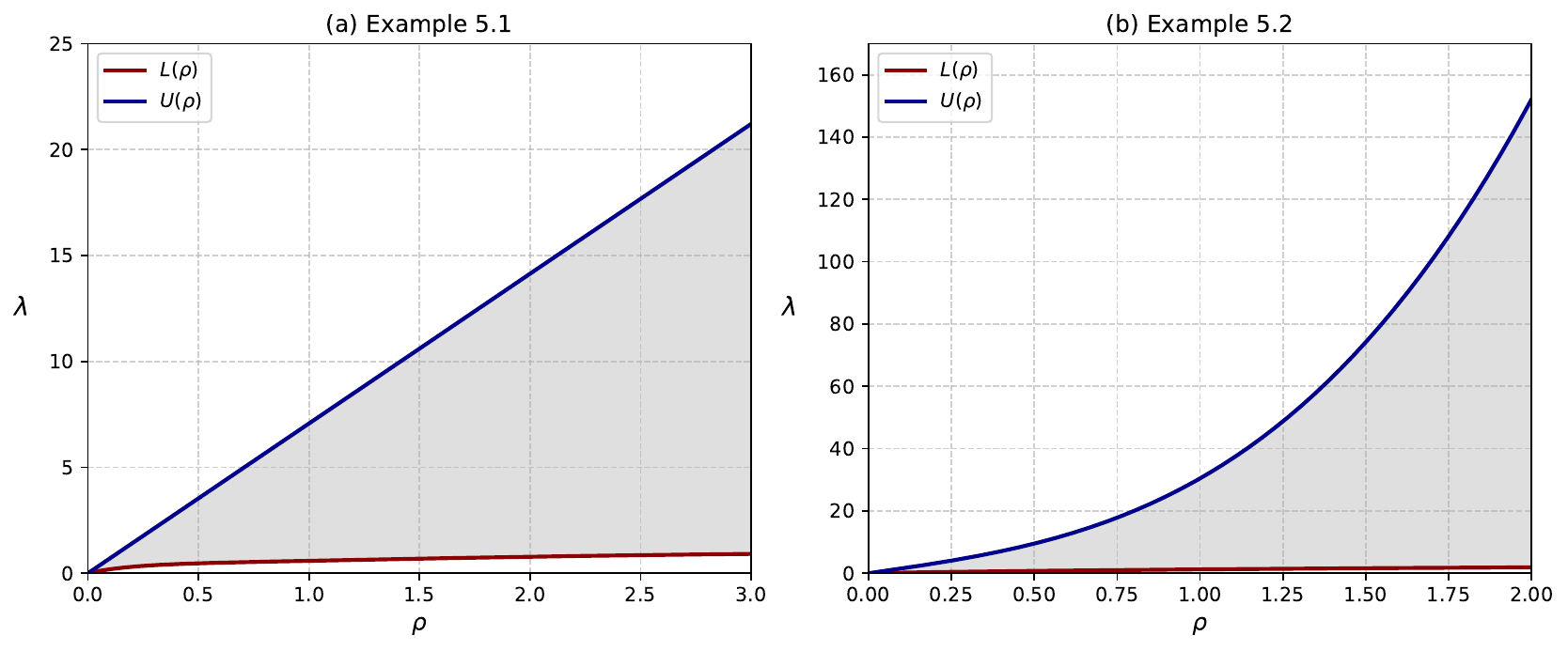}}\hfill
    \subfloat[Sign-changing kernel.]{%
        \includegraphics[width=0.48\textwidth,trim=394bp 0 0 0,clip]{fig_localization.pdf}}
    \caption{Localization of $(\rho,\lambda_{\rho})$: (a) Example 5.1, (b) Example 5.2. In each panel the shaded region lies between $L(\rho)$ and $U(\rho)$.}
    \label{fig:exampleplot_positive}
\end{figure}

\section{Conclusion}
In this paper, we have extended previous work on fractional thermostat models by considering a parameter-dependent boundary value problem involving two nonlinear, non-negative functionals in the boundary conditions. By employing a version of the Birkhoff–Kellogg theorem in cones, we have established the existence of eigenpairs $(u_\rho,\lambda_\rho)$ with a prescribed norm $\|u_\rho\|_\infty=\rho$ in three different settings, depending on the sign of the Green function $K$ and the affine auxiliary function $\gamma$. When both functions change sign, we have shown that we can obtain eigenfunctions that are positive on a suitable subinterval. We have also derived explicit bounds for the localization of the eigenvalues and illustrated our results through two examples. With this approach, we have extended the existing thermostat models to accommodate nonlinear, non-affine controllers in the boundary conditions.
\section*{Acknowledgements}
The authors would like to thank the anonymous referees for their careful reading of the manuscript and constructive comments.
G.~Infante is a member of the Gruppo Nazionale per l'Analisi Matematica, la Probabilit\`a e le loro Applicazioni (GNAMPA) of the Istituto Nazionale di Alta Matematica (INdAM) and the ``The Research ITalian network on Approximation (RITA)''.
This work has been partially drafted during an Erasmus+ stay of T.~Zeghida at the University of Calabria. T.~Zeghida gratefully acknowledges Professors R.~Khaldi and A.~G.~Lakoud for providing the opportunity for this visit and thanks G.~Infante for the collaboration.

\section*{Funding}
G.~Infante has been partially supported by the GNAMPA and acknowledges the Italian Ministry of University and Research's grant [DD 170 24.09.2025 - DM MUR 737_2021] supporting the University of Calabria for financing research in ``Aree Disciplinari Sociali e Umanistiche'', Department of Economics, Statistics and Finance [DD 248/2025 17.12.2025]. Project title: ``Actuarial and decision-making models to support the planning and development of healthcare welfare in Southern Italy and Calabria''.

\section*{Conflicts of interest}
The authors declare no conflict of interest.

\section*{Contribution statement}
All authors contributed equally to this manuscript.

\section*{AI disclosure}
During the preparation and revision of this manuscript,
ChatGPT (OpenAI) was used to assist in checking mathematical
arguments and estimates, identifying potential logical and
notational inconsistencies, and improving the clarity of
the presentation. All AI-assisted suggestions were critically
assessed and the mathematical content independently verified
by the authors, who take full responsibility for the final
version of the manuscript.

\end{document}